\documentclass{amsart}
\usepackage{amssymb}
\usepackage{amsmath}
\usepackage{amsthm}
\usepackage{amscd}
\usepackage{mathrsfs}
\theoremstyle{definition}
\newtheorem{dfn}{Definition}[section]
\theoremstyle{plain}
\newtheorem{thm}[dfn]{Theorem}

\newtheorem{prop}[dfn]{Proposition}
\newtheorem{cor}[dfn]{Corollary}

\theoremstyle{remark}

\newtheorem{rem}[dfn]{Remark}

\title{A short note on strong convergence of $q$-Gaussians}
\author{Akihiro Miyagawa}
\date{June 2023}

\begin{document}

\maketitle
\begin{abstract}
In this note, we prove strong convergence of $q$-Gaussians with respect to a parameter $q$, which implies the spectrum of any self-adjoint non-commutative polynomial in $q$-Gaussians is continuously deformed with respect to $q$. With Bo\.{z}ejko's Haagerup-type inequality, we follow Brannan's approach to proving the asymptotic strong convergence of the free orthogonal quantum group.    
\end{abstract}

\section{Introduction}
The concept of strong convergence has attracted substantial attention in the fields of free probability and Random matrix theory. This interest is largely driven by the tendency of particular multiple random matrices (for instance, independent GUE \cite{HT12} and Haar unitary \cite{CM14}) to demonstrate strong convergence to free random variables when their size goes to infinity. 

As an example of applications, once we can show strong convergence, we can estimate the operator norm of a polynomial in random matrices with sufficiently large sizes by computing that of their limit, which can be applied to show additivity violation of the minimum output entropy in quantum information \cite{Collins16}.
 Moreover, strong convergence of random matrices has applications to operator algebras. An early application is due to Haagerup and Thorbjørnsen \cite{HT12} who prove that $\mathrm{Ext}(C^*_{\mathrm{red}}(\mathbb{F}_2))$ is not a group. Another application was found by Hayes \cite{Hayes20} to reformulate the Peterson-Thom conjecture for free group factors, and it was subsequently proved by Belinschi-Capitaine \cite{BC22}, and Bordenave-Collins \cite{BC23}. 

Strong convergence also appears in group and quantum group theory. In particular, Brannan \cite{Brannan18} proved strong convergence of the free orthogonal quantum groups. His idea is to prove the Haagerup-type inequality which is also known as RD (Rapid Decay) property and to combine it with convergence in non-commutative distribution.   

In this paper, we prove strong convergence of $q$-Gaussians with respect to $q$, based on Brannan's method.
The $q$-Gaussian (also known as $q$-semicircle) distribution is an interpolation between Bernoulli ($q=-1$), semicircle ($q=0$), and Gaussian ($q=1$) with a parameter $-1 \le q \le 1$. This probability measure is classically known as a probability measure whose orthogonal polynomials are $q$-Hermite polynomials. In the non-commutative context, multi-variable $q$-Gaussians are originally introduced by Frisch and Bourret \cite{FB70}, %cite
and Bo\.{z}ejko, K\"{u}mmerer, Speicher (\cite{BKS97}, \cite{BS91}) realized $q$-Gaussians (and $q$-Brownian motion) as field operators on $q$-deformed Fock space, which is an $q$-analogue of Fermionic and Bosonic Fock spaces in quantum physics. 
Although the connections between $q$-Gaussians and quantum physics are not so much known, $q$-Gaussians are well-studied in operator algebra since they generate von Neumann algebras that share many properties with free group factors.

Now, we state our main result of strong convergence of $q$-Gaussians.
\begin{thm}
For any $-1<q_0<1$, strong convergence of $q$-Gaussians $A^{(q)}=(A_1^{(q)},\ldots,A_d^{(q)})$ holds at $q_0$, i.e. for any non-commutative polynomial $P$, 
$$\lim_{q\to q_0}\|P(A^{(q)})\| =\|P(A^{(q_0)})\|.$$
\end{thm}

The key fact for the proof is the Haagerup-type inequality of $q$-Gaussians proved by Bo\.{z}ejko \cite{Bozejko99}. Thanks to this inequality, we can apply Brannan's approach to show strong convergence from convergence in non-commutative distribution. This kind of argument also appears in Pisier's paper \cite[Section 1 and Section 4]{Pisier16}. Thanks to strong convergence and functional calculus, we can also show the convergence of spectrums in the Hausdorff distance. We can generalize our result in a more general setting where operators satisfy ``uniform RD property" and convergence in non-commutative $\ast$-distribution. We state this at the end of this note.  

\section*{Acknowledgement}
The basic idea of this work was found when the author visited UCSD in March. He would like to thank David Jekel and Prof.~{Kemp} for hosting his visit and inspiring discussions. He would like to thank his supervisor Prof.~{Collins} as well as Prof.~{Hayes}, for recommending him to write this note. He acknowledges support from JSPS Research Fellowships for Young Scientists, JSPS KAKENHI Grant Number JP 22J12186.

\section{Preliminary}
In this section, we introduce notations. One can find more details in \cite{Bozejko99}, \cite{BKS97}, \cite{BS91}. 

Let $H$ be a $d$-dimensional Hilbert space with orthonormal basis $e_1,\ldots,e_d$. 
We consider {\em the algebraic Fock space} of $H$ defined by 
$$\mathcal{F}_{\mathrm{alg}}(H) = \bigoplus_{k=0}^{\infty} H^{\otimes k}, $$
where $H^{\otimes 0} = \mathbb{C} e_0$ with the unit vector $e_0$. In order to express the basis of $\mathcal{F}_{\mathrm{alg}}(H)$, we use the word set $[d]^*$ which consists of letters $[d]=\{1,\ldots,d\}$ and the empty word $0$. For each word $w=w_1\cdots w_n \in [d]^*$ with $w_i \in [d]$, we define $e_w:=e_{w_1}\otimes \cdots \otimes e_{w_n}$. Note that $\{e_w\}_{w \in [d]^*}$ forms a linear basis of $\mathcal{F}_{\mathrm{alg}}(H)$. {\em The $q$-Fock space $\mathcal{F}_q(H)$} is the completion of $\mathcal{F}_{\mathrm{alg}}(H)$ with respect to {\em the $q$-inner product} $\langle \  , \ \rangle_q$ defined by 
$$\langle \xi_1 \otimes \cdots \otimes \xi_m , \eta_1 \otimes \cdots \otimes \eta_n \rangle_q = \delta_{m,n} \sum_{\pi \in S_m } q^{\mathrm{inv}(\pi)} \prod_{i=1}^m \langle \xi_i , \eta_{\pi(i)} \rangle_H$$
where $\delta_{m,n}$ is the Kronecker's delta and $S_m$ is the symmetric group with degree $m$ and $\mathrm{inv}(\pi)=\#\{(i,j)\in [m]^2; \ i<j, \ \pi(i)> \pi(j) \}$ is the number of inversions in $\pi$.
\begin{dfn}
{\em The $q$-Gaussians} $A^{(q)}=(A_1^{(q)},\ldots A_d^{(q)})$ are defined by operators on $\mathcal{F}_q(H)$; 
$$A^{(q)}_i = l_i + l_i^* $$
where $l_i$ is the left creation operator defined by $l_i e_w = e_{iw}$. 
\end{dfn}
The von Neumann algebra $W^*(A^{(q)})$ generated by $q$-Gaussians $A^{(q)}$ has a faithful normal tracial state given by 
$$\tau(X) = \langle X e_0, e_0 \rangle_q,$$
which forms a tracial $W^*$-probability space. 
As well as independent Gaussians satisfying the Wick formula, it is known that joint moments of $q$-Gaussians are characterized by pair partitions and the number of crossings.
\begin{thm}[Theorem 1 in \cite{FB70}]
For any $i_1,\ldots, i_n \in [d]$ and $-1\le q\le 1$, we have
$$\tau(A^{(q)}_{i_1} \cdots A^{(q)}_{i_n})=\sum_{\pi \in P_2(n)} q^{\mathrm{cr}(\pi)} \prod_{(k,l)\in \pi} \delta_{i_k,i_l} $$
where $P_2(n)$ is the set of pair partitions on $[n]$ and $\mathrm{cr}$ is the number of crossings.
\end{thm}
From this theorem, we can see convergence of $q$-Gaussians in non-commutative distribution; for $-1\le q_0 \le 1$ and any non-commutative polynomial $P$, we have
$$\lim_{q\to q_0} \tau[P(A^{(q)})] = \tau[P(A^{(q_0)})].$$

Let $L^2(W^*(A^{(q)}),\tau)$ be the Hilbert space obtained from GNS representation of $W^*(A^{(q)})$ with respect to $\tau$. Then it is known that there is an isomorphism $D$ between Hilbert spaces $\mathcal{F}_q(H)$ and $L^2(W^*(A^{(q)}),\tau)$ which maps $\mathcal{F}_{\mathrm{alg}}(H)$ to the algebra $\mathbb{C}\langle A^{(q)} \rangle$ of non-commutative polynomials in $A^{(q)}$. 

{\em The $q$-Wick polynomials} $\{e_w^{(q)}\}_{w \in [d]^*}$ are defined by $e_w^{(q)}=D(e_w)$. Let us denote the operator norm on $B(L^2(W^*(A^{(q)}),\tau))$ by $\| \ \|$. We also use $L^p$-norm $\| \ \|_p$ ($1\le p < \infty)$ defined by $\| X \|_p = \tau[(X^*X)^{\frac{p}{2}}]^{\frac{1}{p}}$. Note that we have 
$$\lim_{p \to \infty} \|X\|_p = \|X\| $$ 
for $X \in W^*(A^{(q)})$ since $\tau$ is a faithful state.
The following inequality, which is proved by Bo\.{z}ejko, is crucial in this paper. 
\begin{thm}[Proposition 2.1 in \cite{Bozejko99}]\label{Haagerup}
For each $k \in \mathbb{N}$ and $-1<q<1$, we have
$$\left\|\sum_{|w|=k} \alpha_w e_w^{(q)} \right\| \le (k+1)C_{|q|}^{\frac{3}{2}} \left \|\sum_{|w|=k} \alpha_w e_w^{(q)}\right\|_2$$
where $\alpha_w \in \mathbb{C}$ and $C_{|q|}^{-1}=\prod_{m=1}^{\infty}(1-|q|^m).$
\end{thm}

\begin{rem}
This kind of inequality is proved first by Haagerup for generators of free groups \cite{Haagerup78}. Bo\.{z}ejko's Haargerup-type inequality for $q$-Gaussians is extended to a more general setting of Coxeter relations in \cite{Krolak02}. 
\end{rem}

\section{The proof of the main theorem}
\begin{thm}\label{main_result}
For any $-1<q_0<1$, strong convergence of $q$-Gaussians $A^{(q)}=(A_1^{(q)},\ldots,A_d^{(q)})$ holds at $q_0$, i.e. for any non-commutative polynomial $P$, 
$$\lim_{q\to q_0}\|P(A^{(q)})\| =\|P(A^{(q_0)})\|.$$
\end{thm}
\begin{proof}
Since we have
$$\|P(A^{(q)})\|_{2n} =\tau\left[(P^*P)^n(A^{(q)})\right]^{\frac{1}{2n}} \le \|P(A^{(q)})\|,$$
  it is obvious from convergence in non-commutative distribution that
$$ \|P(A^{(q_0)})\| \le \liminf_{q\to q_0}\|P(A^{(q)})\|.$$
For the other direction, we will apply Brannan's approach \cite{Brannan18} with Bo\.{z}ejko's Haagerup-type inequality in Theorem \ref{Haagerup}.
Let $P$ be any non-commutative polynomial of degree $m$. Then we can write $P(A^{(q)})=\sum_{k=0}^{m}\sum_{|w|=k}\alpha_w e_w^{(q)}$, and we have
\begin{eqnarray*}
\|P(A^{(q)})\| &\le& \sum_{k=0}^{m} \|\sum_{|w|=k}\alpha_w e_w^{(q)}\| \\
&\le& \sum_{k=0}^{m} (k+1)C_{|q|}^{\frac{3}{2}} \left \|\sum_{|w|=k} \alpha_w e_w^{(q)}\right\|_2 \\
&\le& (m+1) C_{|q|}^{\frac{3}{2}} \sum_{k=0}^{m} \left \|\sum_{|w|=k} \alpha_w e_w^{(q)}\right\|_2 \\
&\le& (m+1)^{\frac{3}{2}} C_{|q|}^{\frac{3}{2}} \left \|\sum_{k=0}^{m}\sum_{|w|=k} \alpha_w e_w^{(q)}\right\|_2 \\
&=& (m+1)^{\frac{3}{2}} C_{|q|}^{\frac{3}{2}} \|P(A^{(q)})\|_2.
\end{eqnarray*}
where we use orthogonality of $e_w^{(q)}$ with respect to word length. Now, we apply this inequality to $(P^*P)^{n}$ which has a degree $2mn$. Then we have
\begin{eqnarray*}
    \|(P^*P)^{n}(A^{(q)})\|=\|P(A^{(q)})\|^{2n} \le (2mn+1)^{\frac{3}{2}} C_{|q|}^{\frac{3}{2}} \|(P^*P)^{n}(A^{(q)})\|_2.
\end{eqnarray*}
By taking a limit $q \to q_0$ after taking the power of $\frac{1}{2n}$ in both sides, we have from the convergence in non-commutative distribution,
$$\limsup_{q\to q_0} \|P(A^{(q)})\| \le (2mn+1)^{\frac{3}{4n}}C_{|q_0|}^{\frac{3}{4n}} \|(P^*P)^{n}(A^{(q_0)})\|_2^{\frac{1}{2n}} \underset{n\to \infty}{\to} \|P(A^{(q_0)})\|.$$
Thus we have strong convergence of $q$-Gaussians.
\end{proof}
\begin{rem}
When $q_0=\pm 1$, our proof does not work. Actually, the $1$-Gaussian is the standard Gaussian, which is unbounded. The $(-1)$-Gaussian is the discrete measure which takes $\pm 1$ with probability $\frac{1}{2}$. For $-1<q<1$, it is known that the $q$-Gaussian has a density function which is supported on $[-\frac{2}{\sqrt{1-q}},\frac{2}{\sqrt{1-q}}]$ (cf. Theorem 1.10 in \cite{BKS97}). Since $\lim_{q\to -1}\frac{2}{\sqrt{1-q}} =\sqrt{2}$, we don't have strong convergence at $q_0=-1$.   
\end{rem}

As a corollary of the main theorem, we obtain the convergence of spectrums of a self-adjoint polynomial in the Hausdorff distance.

\begin{cor}\label{Hausdorff}
For any $-1<q_0<1$ and any self-adjoint polynomial $P$, we have
$$\lim_{q\to q_0} d_H(\sigma[P(A^{(q)})],\sigma[P(A^{(q_0)})]) = 0$$
where $\sigma[P(A^{(q)})]$ is the spectrum of $P(A^{(q)})$ and $d_H(\cdot, \cdot )$ is the Hausdorff distance.
\end{cor}

\begin{proof}
Let $\epsilon > 0$ be given. Since $\sigma[P(A^{(q_0)})]$ is compact, we can take $\{x_k\}_{k=1}^m \subset \sigma[P(A^{(q_0)})]$ such that $\sigma[P(A^{(q_0)})] \subset \bigcup_{k=1}^m N_{\frac{\epsilon}{2}}(x_k)$ where $N_{\frac{\epsilon}{2}}(x_k)$ is the $\frac{\epsilon}{2}$-neighborhood of $x_k$. For each $k=1,\ldots,m$, we take a continuous function $f_k$ on $\mathbb{R}$ such that $0\le f_k \le 1$ and $f_k(x_k)=1$ and $f_k|_{N_{\frac{\epsilon}{2}}(x_k)^c} =0$. Since $\|f_k(P(A^{(q_0)}))\|=1$ for all $k$, we also have $\|f_k(P(A^{(q)}))\|>0$ if $|q-q_0|$ is sufficiently small by Stone–Weierstrass theorem and strong convergence (actually, convergence in non-commutative distribution is enough for this claim). This implies $N_{\frac{\epsilon}{2}}(x_k) \cap \sigma[P(A^{(q)})]\neq \emptyset$ for each $k$ and we have
$$\sigma[P(A^{(q_0)})] \subset \bigcup_{k=1}^m N_{\frac{\epsilon}{2}}(x_k)\subset \sigma[P(A^{(q)})]+(-\epsilon,\epsilon).$$

On the other hand, we take a continuous function $g$ on $\mathbb{R}$ such that $0\le g \le 1$ and $g=0$ on $\sigma[P(A^{(q_0)})]$ and $g=1$ on the complement of $\sigma[P(A^{(q_0)})]+(-\epsilon,\epsilon)$. Since $g(P(A^{(q_0)}))=0$, we similarly have $\|g(P(A^{(q)}))\| < 1$ if $|q-q_0|$ is sufficiently small, which implies
$$\sigma[P(A^{(q)})] \subset \sigma[P(A^{(q_0)})]+(-\epsilon,\epsilon).$$
Therefore, we obtain $d_H(\sigma[P(A^{(q)})],\sigma[P(A^{(q_0)})])<\epsilon$ if $|q-q_0|$ is sufficiently small.
\end{proof}

\iffalse
\begin{cor}
For any $-1<q_0<1$ and any non-commutative polynomial $P$ (not necessarily self-adjoint), we have
$$\lim_{q\to q_0}r\left[P(A^{(q)})\right] =r\left[P(A^{(q_0)})\right]$$
where $r(X)$ denotes the spectral radius of $X$.
\end{cor}
\begin{proof}
 Let us take a closed interval $I \subset (-1,1)$ that contains $q_0$. We consider a sequence of functions $\{f_k\}_{k=1}^{\infty}$ on $I$ defined by $f_k(q)=\|P(A^{(q)}^{2^k}\|^{2^{-k}}$. Note that each $f_k$ is continuous by Theorem \ref{main_result} and the sequence $\{f_k\}_{k=1}^{\infty}$ is monotone decreasing since $f_k(q)=\|P(A^{(q)}^{2^k}\|^{2^{-k}} $   
\end{proof}

\fi

This kind of argument actually holds for any tuples of operators in $C^*$-algebras with faithful states which satisfy ``uniform RD property" and convergence in non-commutative $\ast$-distribution. 
\begin{prop}
Let $(\mathcal{A}_n,\phi_n)_{n\in \mathbb{N}}$ and $(\mathcal{A}_{\infty},\phi_{\infty})$ be couples of a $C^*$-algebra $\mathcal{A}_n, \mathcal{A}_{\infty}$ and a faithful state $\phi_n, \phi_{\infty}$. Let $X^{(n)}=(X_1^{(n)},\ldots,X_d^{(n)})$ be a $d$-tuple of operators (not necessarily self-adjoint) in $\mathcal{A}_n$ for each $n \in \mathbb{N} \cup \{\infty\}$. Assume $(X^{(n)})_{n\in \mathbb{N} \cup \{\infty\}}$ satisfy the following two properties 

\begin{itemize} 
\item ``uniform RD property''; there exists constants $C,D >0$ such that for any $n \in \mathbb{N}$ and non-commutative $\ast$-polynomial $P$, $$\| P(X^{(n)}) \| \le C (\mathrm{deg}P + 1)^{D}\|P(X^{(n)}\|_2. $$ where $\mathrm{deg}P$ is the degree of $P$.\\ 
\item convergence in non-commutative $\ast$-distribution; for any non-commutative $\ast$-polynomial $P$, $$\lim_{n \to \infty}\phi_n[P(X^{(n)})] = \phi_{\infty}[P(X^{(\infty)})].$$
\end{itemize} 
Then, $X^{(n)}$ strongly converges to $X^{(\infty)}$; for any $\ast$-polynomial $P$, $$\lim_{n \to \infty} \|P(X^{(n)})\| = \|P(X^{(\infty)})\|.$$
We also have for any self-adjoint $\ast$-polynomial $P$, 
$$\lim_{n \to \infty} d_H(\sigma[P(X^{(n)})],\sigma[P(X^{(\infty)})])=0.$$
\end{prop}

\begin{proof}As well as the proof of the main theorem, we can estimate $\|(P^*P)^k(X^{(n)})\|$ by using ``uniform RD property". Since $\lim_{k \to \infty}[(2k \cdot \mathrm{deg}P +1 )^D]^{\frac{1}{2k}} =1$, we obtain strong convergence from convergence in non-commutative $\ast$-distribution. By a similar argument in the proof of Corollary \ref{Hausdorff}, we also obtain the convergence of spectrums in the Hausdorff distance.\end{proof}

\begin{rem}
We conclude with a remark on the spectral radius $r[P(X^{(n)})]$ for a non-self-adjoint polynomial $P$. For the same reason that the infimum of continuous functions is upper semi-continuous, we can say by strong convergence (note that $r(T)=\inf_{k}\|T^k\|^{\frac{1}{k}}$ for a bounded operator $T$),
$$\limsup_{n\to \infty} r[P(X^{(n)})] \le r[P(X^{(\infty)})].$$
In particular, the spectral radius $r[P(A^{(q)})]$ is upper semi-continuous with respect to $q$.
We were not able to show the lower semi-continuity. For this problem, a quantitative estimate for the difference between $\|P(A^{(q)})^k\|^{\frac{1}{k}}$ and $\|P(A^{(q')})^k\|^{\frac{1}{k}}$ with different $q,q' \in (-1,1)$ should be helpful. By the Haagerup-type inequality, it is enough to see the difference between $\|P(A^{(q)})^k\|_2^{\frac{1}{k}}$ and $\|P(A^{(q')})^k\|_2^{\frac{1}{k}}$ for sufficiently large $k$. We leave it for future work.
\end{rem}

\end{document}